\documentclass[10pt,a4paper]{amsart}
\usepackage[english]{babel}
\usepackage{amssymb}
\DeclareMathOperator{\cf}{cf}

\newtheorem{theorem}{Theorem}
\newtheorem{fact}[theorem]{Fact}
\newtheorem{definition}[theorem]{Definition}

\newtheorem{lemma}[theorem]{Lemma}
\newtheorem{claim}[theorem]{Claim}
\newtheorem*{theorem*}{Theorem}
\newtheorem*{question*}{Question}
\begin{document}
\title{A note on the normal filters extension property}
\author{Yair Hayut}
\address{Kurt G\"odel Research Center,
Institut f\"ur Mathematik,
Universit\"{a}t Wien, 
Wien 1020,
Austria
}
\thanks{This research was partially supported by the FWF Lise Meitner grant, 2650-N35}
\email{yair.hayut@univie.ac.at}
\begin{abstract}
We show that if $\lambda^{<\kappa}  = \lambda$ and every normal filter on $P_\kappa\lambda$ can be extended to a $\kappa$-complete ultrafilter then so does every $\kappa$-complete filter on $\lambda$. This answers a question of Gitik.
\end{abstract}
\maketitle

In \cite{Gitik2018}, Gitik shows that consistently for every stationary set $X \subseteq \kappa$, there is a $\kappa$-complete filter extending the club filter restricted to $X$, and that the consistency strength of this is a weak repeat point, which is a large cardinal axiom weaker than $o(\kappa)=\kappa^{++}$. In this paper, Gitik asks whether the consistency strength of the stronger statement ``every \emph{normal} filter on $\kappa$ extends to a $\kappa$-complete ultrafilter'' is still below $o(\kappa) = \kappa^{++}$. In this paper, we answer Gitik's question, by showing that the extension property for normal filters is equivalent to the extension property for $\kappa$-complete filters, and in particular it implies failure of squares at successive cardinals and has a high consistency strength.
 
Recall that a filter $\mathcal{F}$ on $P_\kappa\lambda$ is \emph{normal} if for every $A \in \mathcal{F}$ and a choice function $f \in \prod A$ (namely, $f\colon A \to \lambda$, $f(x) \in x$), there is an ordinal $\gamma < \lambda$ such that $f^{-1}(\{\gamma\}) \in \mathcal{F}^+$. 

Equivalently, a filter $\mathcal{F}$ on $P_\kappa\lambda$ is normal if it is closed under diagonal intersections:  for every sequence of sets from $\mathcal{F}$, $\langle A_i \mid i < \lambda\rangle$, 
\[\Delta_i A_i := \{a \in P_\kappa\lambda \mid \forall i \in a,\, a \in A_i\} \in \mathcal{F}.\]

\begin{theorem*}
Let $\kappa\leq \lambda$ be cardinals, where $\cf\kappa = \kappa$ and $\lambda^{<\kappa} = \lambda$. The following are equivalent: 
\begin{itemize}
\item Every normal filter on $P_\kappa \lambda$ can be extended to a $\kappa$-complete ultrafilter.
\item Every $\kappa$-complete filter on $\lambda$ can be extended to a $\kappa$-complete ultrafilter. 
\end{itemize}
\end{theorem*}

In \cite{Hayut2019}, the second assertion in the theorem was shown to be equivalent (for cardinals $\lambda = \lambda^{<\kappa}$) to the $\kappa$-compactness of the generalized logic $\mathcal{L}_{\kappa,\kappa}$ for languages of size $2^\lambda$ which in turn equivalent to the existence of certain elementary embeddings, that behave like local versions for strongly compact embeddings. In particular, assuming that every normal filter on $\kappa$ can be extended to a $\kappa$-complete ultrafilter, both $\square(\kappa)$ and $\square(\kappa^{+})$ fail. Using inner model theoretical results from \cite{Jensen2009StackingMice}, this implies the consistency of a a non-domestic mouse (which is stronger than the consistency of $\mathrm{ZF} + \mathrm{AD}_{\mathbb{R}}$).  

We will use the following definition of clubs at $P_\kappa\lambda$, due to Jech (see \cite{JechHandbook}):
\begin{definition}
Let $\kappa = \cf \kappa \leq \lambda$ be cardinals. A set $C \subseteq P_\kappa\lambda$ is a club if it closed under increasing unions of length $<\kappa$ and for any $x\in P_\kappa \lambda$ there is $y\in C$, $y \supseteq x$.

A set $S\subseteq P_\kappa\lambda$ is stationary if it has a non trivial intersection with every club.
\end{definition}

The club filter is normal and fine (contains the cones $C_\alpha = \{y \in P_\kappa\lambda\mid \alpha \in y\}$ for every $\alpha < \lambda$) and every fine normal filter contains the club filter.

In order to construct the relevant filter in the proof, we will use the following strengthening of independent families.
\begin{definition}
A family $\mathcal{I}$ of subsets of $P_\kappa \lambda$ is \emph{normally independent} if for every pair of disjoint collections $\{A_i \mid i < \lambda\}, \{B_i \mid i < \lambda\} \subseteq \mathcal{I}$, 
$\Delta_{i<\lambda} \left(A_i \setminus B_i\right)$ is stationary.
\end{definition}
In order to construct a normally independent family of , we will assume a guessing principle, $\diamondsuit_{\kappa,\lambda}$, which was defined by Jech in \cite[Section 3]{Jech1972}. We follow the notations of \cite{DonderMatet1993}. 
\begin{definition}
Let $\kappa \leq \lambda$ be cardinals and let $S\subseteq P_\kappa\lambda$. An object $\langle X_a \mid a \in S\rangle$ is a $\diamondsuit_{\kappa,\lambda}(S)$-sequence if for every $Y \subseteq \lambda$ the set $\{a \in S\mid Y \cap a = X_a\}$ is stationary.
\end{definition}
Jech used a two cardinal diamond sequence in order to derive a large almost disjoint collection of stationary subsets of $P_\kappa\lambda$. We will use similar arguments in order to construct a normally independent family.
\begin{lemma}\label{lemma: normally ind}
Let $\kappa \leq \lambda$ be regular cardinals. Assume $\diamondsuit_{\kappa,\lambda}(P_\kappa\lambda)$. There is a normally independent family $\mathcal{I} \subseteq \mathcal{P}(P_\kappa\lambda)$ of cardinality $2^\lambda$. 
\end{lemma}
\begin{proof}
Let $\langle S_a \mid a \in P_\kappa\lambda\rangle$ be a $\diamondsuit_{\kappa,\lambda}(P_\kappa\lambda)$-sequence. Let $p \colon \lambda \times \lambda \times 2 \to \lambda$ be a bijection. Let $a \in P_\kappa\lambda$ be in the club of elements which contain $\{0,1\}$ and $p\restriction a \times a \times 2$ is a bijection between $a \times a \times 2$ and $a$. We may think of $S_a$ as a code for a sequence of pairs of subsets of $a$, $\langle A_a^i, B_a^i\mid i \in a\rangle$, by taking $A^i_a = \{\alpha \in a \mid p(i,\alpha, 0) \in S_a\}$ and $B^i_a = \{\alpha \in a \mid p(i,\alpha, 1) \in S_a\}$. For $a$ which is not in this club, we define $A^i_a = B^i_a = \emptyset$ for all $i\in a$. We abuse notation and write $S_a = \langle A_a^i, B_a^i \mid i \in a\rangle$. 

Let $\langle A_i, B_i \mid i < \lambda\rangle$ be a sequence of subsets of $\lambda$. Then, the set of all $a \in P_\kappa \lambda$ such that $S_a = \langle A_i \cap a, B_i \cap a \mid i \in a\rangle$ is stationary.

Let $A \subseteq \lambda$ and let us define 
\[R_A = \{a \in P_\kappa \lambda \mid \forall i, j \in a,\, A^i_a \neq B^j_a \text { and } \exists i \in a, A \cap a = A^a_i\}.\] 

Let us verify that $\mathcal{I} = \{R_A \mid A \subseteq \lambda\}$ is a normally independent family.  First, note that if $A \neq B$ then $R_A \neq R_B$, and in particular $|\mathcal{I}| = 2^\lambda$. Indeed, for stationary many $a\in P_\kappa\lambda$, $A^i_a = A\cap a$ and $B^i_a = B \cap a$ for all $i \in a$. Taking $a$ large enough so that $A \cap a \neq B \cap a$, we obtain $a \in R_A \setminus R_B$. 

Let us show the normal independence. Let $\{R_{A_i} \mid i < \lambda\}, \{R_{B_i} \mid i < \lambda\}$ be disjoint collections of elements in $\mathcal{I}$.  For every $a\in P_\kappa \lambda$ that guesses the sequence $\langle A_i, B_i \mid i < \lambda\rangle$ correctly and satisfies \[\{A_i \cap a \mid i \in a\} \cap \{B_j \cap a \mid j \in a\} = \emptyset,\]
 $a \in \bigcap_{i\in a} R_{A_i}$ and $a \notin \bigcup_{j\in a} R_{B_j}$.
\end{proof}

\begin{lemma}\label{lemma: diagonal intersection once}
Let $\kappa \leq \lambda$ be regular cardinals. 
Let $\mathcal{F}$ be the minimal normal and fine filter on $P_\kappa\lambda$ that contains a set $\mathcal{A}$. Then $\mathcal{F}$ is non-trivial if and only if for every sequence $\langle A_i \mid i < \lambda\rangle$ of elements of $\mathcal{A}$, $\Delta_{i < \lambda} A_i$ is stationary.
\end{lemma}
\begin{proof}
Since every normal filter contains the club filter, the condition of the lemma is clearly necessary. Let us show that it is sufficient. We claim that 
\[\mathcal{F} = \{ X \subseteq P_\kappa\lambda \mid X \supseteq D \cap \Delta_{i < \lambda} A_i, \,D \text{ is a club, } A_i \in \mathcal{A}\}.\]

Let us show that this collection is closed under diagonal intersections. Let $\langle A^j_i \mid i, j < \lambda\rangle$ be some collection of elements in $\mathcal{A}$ and pick a sequence $\langle B_i \mid i < \lambda\rangle$ that enumerates them. Indeed, the diagonal intersection $\Delta B_i$ and $\Delta_{j < \lambda} \left( \Delta_i A_i^j \right)$ are the same on some club.   
\end{proof}

The last two ingredients for the proof are the following facts:
\begin{fact}[{\cite[Proposition 7.4 and Corollary 5.5]{Matet2015}}]\label{fact: diamond}
If $\kappa$ is subtle then $\diamondsuit_{\kappa,\lambda}(P_\kappa\lambda)$ holds.
\end{fact}

\begin{fact}[{\cite[Theorem 22.17]{Kanamori2008HigherInfinite}}]\label{fact: strong compactness}
There is a fine $\kappa$-complete ultrafilter on $P_\kappa\lambda$ if and only if every $\kappa$-complete filter which is generated by at most $\lambda$ many sets can be extended to a $\kappa$-complete ultrafilter.
\end{fact}

\begin{proof}[Proof of the main theorem]
Note that if the normal filter extension property holds for $\lambda \geq \kappa$ then there is a $\kappa$-complete fine ultrafilter on $P_\kappa \lambda$ (by extending the club filter). Thus, $\kappa$ is measurable and, in particular, subtle. Therefore, by Fact \ref{fact: diamond}, $\diamondsuit_{\kappa,\lambda}(P_\kappa\lambda)$ holds. 

Using Lemma \ref{lemma: normally ind} we can construct a normally independent family of size $2^\lambda$. Let $\mathcal{I} = \{R_X \mid X  \subseteq \lambda\}$ be such a family. For $Y \subseteq \lambda$, we denote $\bar{Y} = \lambda \setminus Y$ and for $Y \subseteq P_\kappa\lambda$, $\bar{Y} = P_\kappa\lambda \setminus Y$.

Let $\mathcal{F}$ be a $\kappa$-complete filter on $\lambda$. We will construct a normal filter $\mathcal{F}'$ such that every extension of $\mathcal{F}'$ to a $\kappa$-complete ultrafilter corresponds to an extension of $\mathcal{F}$ to a $\kappa$-complete ultrafilter. 

Let $\mathcal{A}$ be the following collection of sets: 
\begin{enumerate}
\item $R_X \in \mathcal{A}$ for all $X\in \mathcal{F}$.
\item  For every $X \subset Y \subseteq \lambda$ ($X \neq Y$), $\bar{R}_X \cup R_Y \in \mathcal{A}$.
\item For every $X \subseteq \lambda$, $R_X \cup R_{\bar{X}} \in \mathcal{A}$ and $\bar{R}_{\emptyset} \in \mathcal{A}$. 
\item For every sequence of sets $\langle X_i \mid i < i_\star\rangle$, $i_\star < \kappa$, $R_{\bigcap X_i} \cup \left(\bigcup_{i < i_\star} \bar{R}_{X_i}\right)\in \mathcal{A}$.  
\end{enumerate}

\begin{claim} If $\mathcal{A} \subseteq \mathcal{U}$ is a $\kappa$-complete ultrafilter then $\tilde{\mathcal{U}} = \{X \subseteq \lambda \mid R_X \in \mathcal{U}\}$ is a $\kappa$-complete ultrafilter that extends $\mathcal{F}$.
\end{claim}
\begin{proof}
Let us verify that $\tilde{\mathcal{U}}$ is an ultrafilter. Using (2), for every $X \subseteq Y$, if $R_X \in \mathcal{U}$ then also $R_Y \in \mathcal{U}$, as $\bar{R}_X \cup R_Y \in \mathcal{A}$ and $R_X \cap \left(\bar{R}_X \cup R_Y\right) \subseteq R_Y$. Similarly, using (3), for every $X$, either $X$ or $\lambda \setminus X$ is in $\tilde{\mathcal{U}}$.  Since $R_{\emptyset}\notin \mathcal{U}$, we conclude that $\emptyset \notin \tilde{\mathcal{U}}$.  Finally, we need to show that $\tilde{\mathcal{U}}$ is closed under intersections of length $<\kappa$.  Let $\langle X_i \mid i < i_\star\rangle$ be a sequence of elements of $\tilde{\mathcal{U}}$ and $i_\star < \kappa$. Then by the $\kappa$-completeness of $\mathcal{U}$, $\bigcap R_{X_i} \in \mathcal U$. Therefore, using (4), $R_{\bigcap X_i} \in \mathcal{U}$, and thus $\bigcap X_i \in \tilde{\mathcal{U}}$. 
\end{proof}

Every element of $\mathcal{A}$ except $\bar{R}_{\emptyset}$ is of one of the following forms: $R_X$, $\bar{R}_X \cup R_Y$, $R_X  \cup R_{\bar{X}}$ or $R_{\bigcap X_i} \cup \left(\bigcup_{i < i_\star} \bar{R}_{X_i}\right)$. For $A \in \mathcal{A}$ there is a unique representation, by the independence of $\mathcal{I}$, and we let $d(A)$ be the  collection of $X \subseteq \lambda$ such that $R_X$ appears in the representation of $A$. Explicitly, for $A = \bar{R}_{\emptyset}$, $d(A) = \{\emptyset\}$, for $A = R_X$, $d(A) = \{X\}$, for $A = \bar{R}_X \cup R_Y$ $d(A) = \{X, Y\}$, for $A = R_X  \cup R_{\bar{X}}$, $d(A) = \{X, \bar{X}\}$ and for $A = R_{\bigcap X_i} \cup \left(\bigcup_{i < i_\star} \bar{R}_{X_i}\right)$, $d(A) = \{X_i \mid i < i_\star\} \cup \{ \bigcap X_i\}$. 

Let $\mathcal{F}'$ be the minimal normal filter extending $\mathcal{A}$. By Lemma \ref{lemma: diagonal intersection once}, in order to verify that $\mathcal{F}'$ is a proper filter it is sufficient to verify that every diagonal intersection of elements from $\mathcal{A}$ is stationary.

Let $\langle B_i \mid i < \lambda\rangle$ be a list of elements in $\mathcal{A}$. Let $\mathcal{D} = \bigcup_{i < \lambda} d(B_i)$ and $\mathcal{B} = \mathcal{D} \cap \mathcal{F}$. Clearly, $ |\mathcal{B}| \leq |\mathcal{D}| \leq\lambda$. By Fact \ref{fact: strong compactness} and the existence of a fine $\kappa$-complete ultrafilter on $P_\kappa\lambda$, there is a $\kappa$-complete ultrafilter $\mathcal{U}'$ that extends the $\kappa$-complete filter generated by $\mathcal{B}$.  Let $\langle X_i \mid i < \lambda\rangle$ enumerate $\mathcal{D}$ (with repetitions, if needed).

\begin{claim}\label{claim:di}
For each $i < \lambda$, if $D_i = d(B_i)$ then 
\[\bigcap_{X \in D_i \cap \mathcal{U}'} R_X  \setminus \bigcup_{Y \in D_i \setminus \mathcal{U}'} R_Y = \bigcap_{X \in D_i \cap \mathcal{U}'} R_X  \cap \bigcap_{Y \in D_i \setminus \mathcal{U}'} \bar{R}_Y \subseteq B_i.\]
\end{claim}
\begin{proof}
We split into cases, according to the representation of $B_i$:
\begin{enumerate}
\item $B_i = R_X$ for $X \in \mathcal{F}$. In this case $X \in \mathcal{B} \subseteq \mathcal{U}'$.
\item $B_i = \bar{R}_X \cup R_Y$ for $X \subseteq Y$, then either $X \in \mathcal{U}'$ and then so is $Y$, or that $Y \notin \mathcal{U}'$ and then also $X\notin \mathcal{U}'$. 
\item $B_i = R_X \cup R_{\bar{X}}$. In this case, either $X \in \mathcal{U}'$ or $\bar{X} \in \mathcal{U}'$. If $B_i = \bar{R}_\emptyset$, then $\emptyset \notin \mathcal{U}'$.
\item $B_i = R_{\bigcap X_j} \cup \left(\bigcup_{j < j_\star} \bar{R}_{X_j}\right)$. In this case we use the $\kappa$-completeness of $\mathcal{U}'$. If $\bigcap X_j \notin \mathcal{U}'$ then there must be some $j$ such that $X_j \notin \mathcal{U}'$.   
\end{enumerate} 

\end{proof}
Let us consider 
$S = \Delta_{X_i \in \mathcal{U}'} R_{X_i} \setminus \nabla_{X_j \notin \mathcal{U}'} R_{X_j}$. 
By the normal independence of $\mathcal{I}$, $S$ is stationary. Let $D$ be the club of all $a\in P_\kappa\lambda$ such that for all $\alpha \in a$, if $d(B_\alpha) = \{X_{\beta_i} \mid i < i_\star\}$ then $\beta_i\in a$ for all $i < i_\star$. 

For every $a \in S$, let $C_a$ be
\[\bigcap_{\alpha \in a,\, X_\alpha \in \mathcal{U}'} R_{X_\alpha} \setminus \bigcup_{\beta\in a,\, X_\beta \notin \mathcal{U}'} R_{X_\beta},\]
by the definition of $S$, $a \in C_a$. If we further assume that $a \in D$, then $C_a \subseteq B_i$ for all $i \in a$ by Claim \ref{claim:di}. Thus, $S \cap D \subseteq \Delta_i B_i$, as wanted. 
\end{proof}

Note that even for $\kappa \leq \lambda <\lambda^{<\kappa}$, 
the normal filter extension property on $P_\kappa\lambda$ implies the filter extension property for filters on $\lambda$, since Lemma \ref{lemma: normally ind} still holds.  

The other direction might fail. For example, consider the case $\lambda = \kappa^{+\omega}$ and assume $\mathrm{GCH}$. In order to extend filters on $\lambda$ one only need to consider $2^\lambda = \kappa^{+\omega+1}$ many sets, while in order to extend filters on $P_\kappa \lambda$ one has to consider $\kappa^{+\omega+2}$ many sets. In particular, in a model of $\mathrm{GCH}$ and level-by-level correspondence between supercompact and strong compact (such as \cite{ApterShelah1997}), the least cardinal $\kappa$ for which the club filter on $P_{\kappa}{\kappa^{+\omega}}$ can be extended to a $\kappa$-complete ultrafilter is $\kappa^{+\omega+1}$-supercompact. $\kappa$ is strictly above the least cardinal $\mu$ for which every $\mu$-complete filter on $\mu^{+\omega}$ can be extended to a $\mu$-complete ultrafilter, by a reflection argument: Let $j \colon V \to M$ be a $\kappa^{+\omega+1}$-supercompact embedding. Then since in $V$, every $\kappa$-complete filter on $\kappa^{+\omega}$ can be extended to an ultrafilter then the same holds in $M$, using its closure under $\kappa^{+\omega+1} = 2^{\kappa^{+\omega}}$-sequences. Therefore, the least cardinal $\mu$ such that every $\mu$-complete filter on $\mu^{+\omega}$ can be extended to a $\mu$-complete ultrafilter is below $\kappa$.

\section*{acknowledgements}
I would like to thank the anonymous referee for their useful suggestions for improving this paper.

\providecommand{\bysame}{\leavevmode\hbox to3em{\hrulefill}\thinspace}
\providecommand{\MR}{\relax\ifhmode\unskip\space\fi MR }
\providecommand{\MRhref}[2]{%
  \href{http://www.ams.org/mathscinet-getitem?mr=#1}{#2}
}
\providecommand{\href}[2]{#2}

\end{document}